\documentclass[12pt]{amsart}
\usepackage{amssymb}
\newcommand{\bn}[1]{\[#1\]}
\newcommand{\be}[2]{\begin{equation}\label{#1} {#2}\end{equation}}

\newtheorem{lemme}{Lemma}

\newtheorem{theorem}{Theorem}
\newtheorem{remark}{Remark}

\begin{document}

\title [ Monotonicity  of $q$--Kummer confluent hypergeometric and $q$--hypergeometric functions ]{ Monotonicity of ratios of $q$--Kummer confluent hypergeometric and $q$--hypergeometric functions and associated Tur\'an types inequalities \\}%

\author[ K. Mehrez,\; S.M. Sitnik]{ Khaled Mehrez,\; Sergei M. Sitnik }

 \address{Khaled Mehrez. D\'epartement de Math\'ematiques IPEIM. Monastir 5000, Tunisia.}
 \email{k.mehrez@yahoo.fr}
 \address{Sergei M. Sitnik. Voronezh Institute of the Russian Ministry of Internal Affairs.
Voronezh, Russia.}
\email{pochtaname@gmail.com}

\begin{abstract}
In this paper we prove monotonicity of some ratios of $q$--Kummer confluent hypergeometric and $q$--hypergeometric functions. The results are also closely connected with Tur\'an type inequalities. In order to obtain  main results we apply  methods developed for the  case of classical Kummer and Gauss hypergeometric functions  in \cite{MS1}-\cite{MS2}.
\end{abstract}
\maketitle
\section{\textbf{Introduction}}
In 1941 while studying the zeros of Legendre polynomials the Hungarian mathematician Paul
Tur\'{a}n discovered the following inequality
\[
P_{n-1}(x)P_{n+1}(x)\le\left[P_{n}(x)\right]^{2},\]
where $|x|\le 1,\;n\in\mathbb{N}=\{1,2,...\}$ and $P_n$ stands  for  the  classical  Legendre  polynomial.   This
inequality was published by P. Tur\'an only in 1950 in \cite{T50}.  However, since the publication in 1948
by G. Szeg\H{o} \cite{S48} of the above famous Tur\'{a}n inequality for Legendre polynomials many authors
have deduced analogous results for classical  polynomials and special functions. It has been shown by several researchers that the most important polynomials (e.g.   Laguerre,  Hermite,  Appell,  Bernoulli,  Jacobi,  Jensen, Pollaczek,  Lommel, Askey--Wilson, ultraspherical) and special functions (e.g.  Bessel, modified Bessel,
gamma, polygamma, Riemann zeta) satisfy  Tur\'an type inequalities.  In 1981 one of the
PhD students of P. Tur\'an, L. Alp\'ar \cite{Alp} in Tur\'an's biography mentioned that the above Tur\'an
inequality had a wide--ranging effect, this inequality was dealt with in more than 60 papers. Also  Tur\'an type inequalities  are closely connected with log--convexity and log--concavity of hypergeometric--like functions, cf. \cite{Karp1}-\cite{Karp2}. A survey of recent results on Tur\'an type inequalities \cite{C} is published in the proceedings of the conference \cite{Proc} dedicated to Paul Tur\'an's achievements in different areas of mathematics and applications.

Since Tur\'an's inequality was first
investigated for  orthogonal polynomials in hypergeometric representation afterwards such inequalities
were extensively studied  for various hypergeometric functions as well, e.g. in \cite{arb} Tur\'an type inequalities for the q--Kummer and q--hypergeometric functions were proved.

In \cite{MS1}-\cite{MS2} in terms of monotonicity of ratios of Kummer, Gauss and generalized hypergeometric functions the authors  presented some new Tur\'an type inequalities. They are connected with  problems having some history.

Let us consider the series for the exponential function
$$
\exp(x)=e^x=\sum_{k=0}^{\infty}\frac{x^k}{k!}, \ x\ge 0,
$$
its section $S_n(x)$ and series remainder $ R_n(x)$ in the form
\be{def}{S_n(x)= \sum_{k=0}^{n}\frac{x^k}{k!},
\ R_n(x)=\exp(x)-S_n(x)=\sum_{k=n+1}^{\infty}\frac{x^k}{k!}, \,x\ge 0.}

Besides simplicity and elementary nature of these functions many mathematicians studied problems for them.
G.~Szeg\H{o} proved a remarkable limit distribution for zeroes of sections, accumulated along so--called   the Szeg\H{o} curve (\cite{ESV}). S.~Ramanujan seems was the first who proved the non--trivial  inequality for exponential sections in the form (\cite{Ram}, pp. 323--324) : if
$$
\frac{e^n}{2}=R_{n-1}(n)+\frac{n^n}{n!}\theta(n)
$$
then
$$
\frac{1}{3} < \theta(n)=\frac{n!\left(\frac{e^n}{2} - R_{n-1}(n)\right)}{n^n} < \frac{1}{2}.
$$
This result is important as it also leads to explicit  rational bounds for $e^n$
as it was specially pointed out in (\cite{Ram}, pp. 323--324).

In the preprint  \cite{Sit1} in 1993 were thoroughly studied inequalities of the form
\be{inq}{m(n)\le f_n(x)=\frac{R_{n-1}(x)R_{n+1}(x)}{\left[R_n(x)\right]^2} \le M(n), \,x\ge 0.}
The search for the best constants $m(n)=m_{best}(n),\,M(n)=M_{best}(n)$ has some history. The left--hand side of (\ref{inq}) was first proved by Kesava Menon in \cite{KM} with
$m(n)=\frac{1}{2}$ (not best) and by Horst Alzer in \cite{ALZ} with
\begin{equation}\label{mbest}
m_{best}(n)=\frac{n+1}{n+2}=f_n(0),
\end{equation}
cf. \cite{Sit1} for the more detailed history.
In  \cite{Sit1} it was also shown that in fact the inequality (\ref{inq}) with  the sharp lower constant (\ref{mbest}) is a special case of the stronger inequality proved earlier in 1982 by Walter Gautschi  in \cite{Gau1}.

It seems that the right--hand side of (\ref{inq}) was first proved by the author in \cite{Sit1} with $M_{best}=1=f_n(\infty)$. In \cite{Sit1} dozens of generalizations of inequality   (\ref{inq}) and related results were proved. May be in fact it was the first example of so called Turan--type inequality for special case of the Kummer hypergeometric functions.

Obviously the above inequalities are consequences of the next conjecture  originally formulated in \cite{Sit1} and recently revived in \cite{Sita1}--\cite{Sita2}.

\textit{\textbf{Conjecture 1.} The function $f_n(x)$ in  (\ref{inq}) is monotone increasing\\ for $x\in [0; \infty), n\in \mathbb{N}$.} So the next inequality is valid
\be{mon}{\frac{n+1}{n+2}=f_n(0) \le f_n(x) < 1=f_n(\infty).}

In 1990's we tried to prove this conjecture in the straightforward manner  by expanding an inequality  $(f_n(x))^{'}\ge 0$ in series and multiplying  triple products of  hypergeometric functions but failed (\cite{Sit2}--\cite{Sita1}).

Consider a representation via Kummer hypergeometric functions
\be{Kum}{f_n(x)=\frac{n+1}{n+2} \ g_n(x), \ g_n(x)=\frac{_1F_1(1; n+1; x)  _1F_1(1; n+3; x)}{\left[_1F_1(1; n+2; x)\right]^2}.}
So the conjecture 1 may be reformulated in terms of this function $g_n(x)$ as conjecture 2.

\textit{\textbf{Conjecture 2.} The function $g_n(x)$ in  (\ref{Kum}) is monotone increasing\\ for $x\in [0; \infty), n\in \mathbb{N}$.}

This leads us to the next more general

\textit{\textbf{Problem 1.} Find monotonicity  in $x$ conditions for $x\in [0; \infty)$\\ for all parameters a,b,c for the function}
\be{prob}{h(a,b,c,x)=\frac{_1F_1(a; b-c; x)  _1F_1(a; b+c; x)}{\left[_1F_1(a; b; x)\right]^2}.}

We may also call (\ref{prob}) mockingly (in Ramanujan way, remember his mock theta--functions!) "\textit{The abc--problem}" for Kummer hypergeometric functions, why not?

Another generalization is to change Kummer  hypergeometric functions to higher ones.

\textit{\textbf{Problem 2.} Find monotonicity  in $x$ conditions for $x\in [0; \infty)$ for all\\ vector--valued parameters a,b,c for the function}
\be{probpq}{h_{p,q}(a,b,c,x)=\frac{_pF_q(a; b-c; x)  _pF_q(a; b+c; x)}{\left[_pF_q(a; b; x)\right]^2},}
\bn{a=(a_1,\ldots, a_p), b=(b_1,\ldots, b_q), c=(c_1, \ldots, c_q).}

This is "\textit{The abc--problem}" for generalized hypergeometric functions. The more complicated problems are obvious and may be considered for pairs or triplets of parameters and also for multivariable hypergeometric functions.

Recently the above problems 1,2 and conjectures 1,2 were proved by the authors \cite{MS1}--\cite{MS2}.
In this paper we prove $q$--versions of these results for the classical Kummer and Gauss hypergeometric functions, cf. also \cite{MS3}.

Next let us recall the following results which will be used in the sequel.
\begin{lemme}\label{l1}
Let $(a_{n})$ and $(b_{n})$ $(n=0,1,2...)$ be real numbers, such that $b_{n}>0,\;n=0,1,2,...$ and $\left(\frac{a_{n}}{b_{n}}\right)_{n\geq 0}$ is increasing (decreasing), then $\left(\frac{a_{0+}...+a_{n}}{b_{0}+...+b_{n}}\right)_{n}$ is also increasing (decreasing).
\end{lemme}

\begin{lemme}\label{l2}(cf. \cite{BK}--\cite{PV}).
Let $(a_{n})$ and $(b_{n})$ $(n=0,1,2...)$ be real numbers and
let the power series $A(x)=\sum_{n=0}^{\infty}a_{n}x^{n}$ and $B(x)=\sum_{n=0}^{\infty}b_{n}x^{n}$
be convergent if $|x|<r$. If $b_{n}>0,\, n=0,1,2,...$ and if the
sequence $\left(\frac{a_{n}}{b_{n}}\right)_{n\geq0}$is (strictly)
increasing (decreasing) , then the function $\frac{A(x)}{B(x)}$ is also
(strictly) increasing on $[0,r[$.
\end{lemme}

\section{\textbf{Notations and preliminaries}}
Throughout this paper, we fix $q\in]0,1[.$ We refer to \cite{kac}, \cite{olbc} and \cite{gr} for the definitions, notations and properies of  the $q$--shifted factorials and $q$--hypergeometric functions.

\subsection{Basic symbols}

Let $a\in\mathbb{R}$ then $q$--shifted factorials are defined by
\[
(a;q)_{0}=1,\,\,\,\,\,(a;q)_{n}=\prod_{k=0}^{n-1}(1-aq^{k}),\,\,\,(a;q)_{\infty}=\prod_{k=0}^{\infty}(1-aq^{k}),\]
and we write
$$(a_{1},a_{2},...,a_{p};q)=(a_{1};q)_{n}(a_{1};q)_{n}...(a_{p};q)_{n},\;n=0,1,2,...$$
Note that for $q\longrightarrow1$ the expression $\frac{(q^{a};q)_{n}}{(1-q)^{n}}$ tends to $(a)_{n}=a(a+1)...(a+n-1).$

\subsection{$q$-Kummer confluent hypergeometric functions}
The $q$--Kummer confluent hypergeometric function is defined by
\begin{equation}
\phi(q^{a},q^{c};q,x)=_{1}\phi_{1}(q^{a},q^{c};q,(1-q)x)=\sum_{n\geq0}\frac{(q^{a};q)_{n}(1-q)^{n}}{(q^{c};q)(q;q)_{n}}x^{n},
\end{equation}
for all $a,c\in\mathbb{R}$ and $x>0,$ which for $q\longrightarrow1$ is reduced to the Kummer confluent hypergeometric function
$$_{1}F_{1}(a;c;x)=\sum_{n=0}^{\infty}\frac{(a)_{n}}{(c)_{n}n!}x^n.$$

\subsection{$q$- hypergeometric functions}
The $q$--hypergeometric series or basic hypergeometric series is defined by \cite{olbc},\cite{gr}
\begin{eqnarray}\label{kh}
_{p}\Phi_{r}(a_{1},...,a_{p};b_{1},...,b_{r};q;x)=\nonumber\\
=\sum_{n=0}^{\infty}\frac{(a_{1};q)_{n}(a_{2};q)_{n}...(a_{p};q)_{n}}{(b_{1};q)_{n}(b_{2};q)_{n}...(b_{r},q)_{n}(q;q)_{n}}\left[(-1)^{n}q^{(_{2}^{n})}\right]^{1+r-p}x^{n}
\end{eqnarray}
with $(_{2}^{n})=\frac{n(n-1)}{2},\;\;a_{k},b_{k}\in\mathbb{R}\in\mathbb{C},\;b_{k}\neq q^{-n},\,\, k=1,...,r,\, n\in\mathbb{N}_{0},\,\,0<|q|<1.$
The left hand side of (\ref{kh})  represents the $q$--hypergeometric function ${}_{p}\phi_{r}$ where the series converges. Assuming $0<|q|<1,$ the following conditions are valid for the convergence of  (\ref{kh}) (cf.\cite{gr}).
\begin{itemize}
\item $p<r+1$: the series converges absolutely for $x\in\mathbb{C},$
\item $p=r+1:$ the series converges for $|x|<1,$
\item $p>r+1:$ the series converges only for $x=0$, unless it terminates.
\end{itemize}

Since for $q\longrightarrow1$ the expression $\frac{(q^{a};q)_{n}}{(1-q)^{n}}$ tends to $(a)_{n}=a(a+1)...(a+n-1),$ we evaluate
\[
\lim_{q\longrightarrow1}{}_{p}\Phi_{r}(q^{a_{1}},...,q^{a_{p}};q^{b_{1}},...,q^{b_{r}};q;x)=_{p}F_{r}(a_{1},...,a_{p};b_{1},...,b_{r};x)=\sum_{n=0}^{\infty}\frac{(a_{1})_{n}...(a_{p})_{n}}{(b_{1})_{n}...(b_{r})_{n}n!}x^{n},\]
where $_{p}F_{r}$ stands for the generalized hypergeometric function.

\section{\textbf{Monotonicity of ratios of $q$--Kummer hypergeometric functions}}
In this section we consider the function
\begin{equation}
h(a,b,c,q,x)=\frac{\phi(q^{a},q^{b-c},q,x)\phi(q^{a},q^{b+c},q,x)}{\left[\phi(q^{a},q^{b},q,x)\right]^{2}},
\end{equation}
for all $a,b\in\mathbb{R}\; \textrm{and}\; x>0. $
The following theorem is the $q$--version of the theorem 1 from (\cite{MS1}--\cite{MS2}).
\begin{theorem}\label{tt1}
Let $q\in]0,1[,$ if $a>b>c>0,\, b>1,$ then the function $x\longmapsto h(a,b,c,q,x)$
is increasing on $[0,\infty[.$ In particular, the following Tur\'an
type inequality is valid for all $a>b>c>0,\, b>1$ and $q\in]0,1[$
\begin{equation}\label{t1}
\left[\phi(q^{a},q^{b},q,x)\right]^{2}\leq\phi(q^{a},q^{b-c},q,x)\phi(q^{a},q^{b+c},q,x).
\end{equation}
\end{theorem}
\begin{proof}For convenience let us write $\phi(q^{a},q^{b},q,x)$ as
$$\phi(q^{a},q^{b},q,x)=\sum_{n=0}^{\infty}u_{n}(a,b,q)x^{n},$$
where
$$u_{n}(a,b,q)=\frac{(q^{a};q)_{n}(1-q)^n}{(q^{b};q)_{n}(q;q)_n}.$$
Then
\begin{displaymath}
\begin{split}
h(a,b,c,q,x)&=\frac{\left(\sum_{n=0}^{\infty}u_{n}(a,b-c,q)x^{n}\right)
\left(\sum_{n=0}^{\infty}u_{n}(a,b-c,q)x^{n}\right)}{\left(\sum_{n=0}^{\infty}u_{n}(a,b,q)x^{n}\right)^{2}}=\\
&=\frac{\sum_{n=0}^{\infty}v_{n}(a,b,c,q)x^{n}}{\sum_{n=0}^{\infty}w_{n}(a,b,q)x^{n}},
\end{split}
\end{displaymath}
with\\
$$v_{n}(a,b,c,q)=\displaystyle{\sum_{k=0}^{n}u_{k}(a,b-c,q)u_{n-k}(a,b+c,q)}$$
and
$$w_{n}(a,b,q)=\sum_{k=0}^{n}u_{k}(a,b,q)u_{n-k}(a,b,q).$$
Let define the sequences $(A_{n,k})_{k\geq0}$ by
$$A_{n,k}(a,b,c,q)=\frac{u_{k}(a,b-c,q)u_{n-k}(a,b+c,q)}{u_{k}(a,b,q)u_{n-k}(a,b,q)}=
\frac{(q^{b};q)_{k}(q^{b};q)_{n-k}}{(q^{b-c};q)_{k}(q^{b+c};q)_{n-k}}$$
and evaluate
\begin{equation*}
\begin{split}
\frac{A_{n,k+1}(a,b,c,q)}{A_{n,k}(a,b,c,q)}&=\frac{(q^{b};q)_{k+1}(q^{b};q)_{n-k-1}(q^{b-c};q)_{k}(q^{b+c};q)_{n-k}}
{(q^{b-c};q)_{k+1}(q^{b+c};q)_{n-k-1}(q^{b};q)_{k}(q^{b};q)_{n-k}}=\\
&=\left(\frac{(q^{b};q)_{k+1}}{(q^{b};q)_{k}}\right).\left(\frac{(q^{b-c};q)_{k}}{(q^{b-c};q)_{k+1}}\right).
\left(\frac{(q^{b};q)_{n-k-1}}{(q^{b};q)_{n-k}}\right).\left(\frac{(q^{b+c};q)_{n-k}}{(q^{b+c};q)_{n-k}}\right)=\\
&=\left(\frac{1-q^{b+k}}{1-q^{b-c+k}}\right).\left(\frac{1-q^{b+c+n-k-1}}{1-q^{b+n-k-1}}\right).
\end{split}
\end{equation*}
Since $q\in]0,1[$ and $b>1$ it follows $\frac{A_{n,k+1}(a,b,c,q)}{A_{n,k}(a,b,c,q)}\geq1$ and consequently the sequence $(A_{n,k}(a,b,c,q))_{k\geq0}$ is increasing. We conclude that $C_n$ defined by $C_{n}=\frac{u_{n}}{v_{n}}$ is increasing  by Lemma \ref{l1}. Thus the function $x\longmapsto h(a,b,c,q,x)$ is increasing on $[0,\infty[$ by Lemma \ref{l2}.\\
Furthermore,
\[\lim_{x\longrightarrow0}h(a,b,q,x)=1,\]
and the Tur\'an type inequality (\ref{t1}) follows. So the proof of Theorem \ref{tt1} is complete.
\end{proof}

\begin{remark} The inequality (\ref{t1}) is interesting as a consequence of monotonicity property we consider. This inequality itself is not new and may be found in \cite{arb}.
\end{remark}

\section{\textbf{Monotonicity of ratios of $q$--hypergeometric functions}}
In this section we consider the function $h_{r}(a,b,c,q)$ defined by
\begin{eqnarray}\label{K4}
h_{r}(a,b,c,q)=\nonumber\\
\ \ \frac{\phi(q^{a_{1}},..,q^{a{_{r+1}}};q^{b_{1}-c_{1}},...,q^{b_{r}-c_{r}};q,x)
\phi(q^{a_{1}},..,q^{a{_{r+1}}};q^{b_{1}-c_{1}},...,q^{b_{r}-c_{r}};q,x)}
{\left[\phi(q^{a_1},..,q^{a_{r+1}};q^{b_1},...,q^{b_r};q,x)\right]^{2}}
\end{eqnarray}
where $a=(a_{1},...,a_{r+1})\;b=(b_{1},...,b_{r})$ and $c=(c_{1},...,c_{r})$ for all $a_{k},b_{k},c_{k}\in\mathbb{R},\;b_{k}\neq q^{-n},\,\, k=1,...,r,\, n\in\mathbb{N}_{0},\,\,0<|q|<1.$
\begin{theorem} Let $r\in\mathbb{N}\;q\in(0,1)\;a=(a_{0},...,a_r),\;b=(b_{1},...,b_{r})\;c=(c_1,...,c_r),\;b_i>c_i$ for $i=1,...,r.$ If $b_i>1$ for $i=1,...,r$, then the function $h_r(a,b,c,q)$ is strictly increasing on $[0,1[.$ Moreover, if  $b_i>c_i, b_i>1$, and $q\in(0,1)$, then the next Tur\'an type inequality holds
\begin{eqnarray}\label{110}
\left[\phi(q^{a_1},..,q^{a_{r+1}};q^{b_1},...,q^{b_r};q,x)\right]^{2}<\nonumber\\
\phi(q^{a_{1}},..,q^{a{_{r+1}}};q^{b_{1}-c_{1}},...,q^{b_{p}-c_{p}};q,x)
\phi(q^{a_{1}},..,q^{a{_{p+1}}};q^{b_{1}-c_{1}},...,q^{b_{r}-c_{r}};q,x).
\end{eqnarray}
\end{theorem}
\begin{proof}
By using the inequality (\ref{K4}), we can write $h_r$ in the form
\begin{equation}
\begin{split}
h_{r}(a,b,q,x)&=\frac{\left(\sum_{n=0}^{\infty}
\frac{(q^{a_{1}};q)_{n}...(q^{a_{r+1}};q)_{n}x^{n}}{(q^{b_{1}-c_{1}};q)_{n}...(q^{b_{r}-c_{r}};q)_{n}(q;q)_{n}}
\right)}
{\left(\sum\frac{(q^{a_{1}};q)_{n}...(q^{a_{r+1}};q)_{n}x^{n}}{(q^{b_{1}};q)_{n}...(q^{b_{r}};q)_{n}(q;q)_{n}}
\right)^{2}}\cdot\\
&.\left(\sum_{n=0}^{\infty}\frac{(q^{a_{1}};q)_{n}...(q^{a_{r+1}};q)_{n}x^{n}}
{(q^{b_{1+c_{1}}};q)_{n}...(q^{b_{r}+c_{r}};q)_{n}(q;q)_{n}}\right)=\\
&=\frac{\sum_{n=0}^{\infty}A_{n}(a,b,c,q)}{\sum_{n=0}^{\infty}B_{n}(a,b,c,q)}x^{n},
\end{split}
\end{equation}
with use of the next notations
\begin{equation*}
\begin{split}
A_{n}(a,b,c,q)&=\sum_{k=0}^{n}U_{k}(a,b,c,q)=\\
&=\sum_{k=0}^{n}\frac{\prod_{j=1}^{r+1}(q^{a_{j}};q)_{n-k}(q^{a_{j}};q)_{k}}{(q;q)_{k}(q;q)_{n-k}\prod_{j=1}^{r}(q^{b_{j}-c_{j}};q)_{k}(q^{b_{j}+c_{j}};q)_{n-k}}
\end{split}
\end{equation*}
and
\begin{equation*}
\begin{split}
B_{n}(a,b,c,q)&=\sum_{k=0}^{n}V_{k}(a,b,c,q)=\\
&=\sum_{k=0}^{n}\frac{\prod_{j=1}^{r+1}(q^{a_{j}};q)_{n-k}(q^{a_{j}};q)_{k}}{(q;q)_{k}(q;q)_{n-k}\prod_{j=1}^{r}(q^{b_{j}};q)_{k}(q^{b_{j}};q)_{n-k}}.
\end{split}
\end{equation*}
For fixed $n\in\mathbb{N}$ we define the sequence $(W_{n,k}(a,b,c,q))_{k\geq0}$ by
\begin{equation*}
\begin{split}
W_{n,k}(a,b,c,q)&=\frac{U_{k}(a,b,c,q)}{V_{k}(a,b,c,q)}=\\
&=\prod_{j=1}^{r}\frac{(q^{b_{j}};q)_{k}(q^{b_{j}};q)_{n-k}}{(q^{b_{j}-c_{j}};q)_{k}(q^{b_{j}+c_{j}};q)_{n-k}}.
\end{split}
\end{equation*}
For $n,k\in\mathbb{N}$ we evaluate
\begin{equation*}
\begin{split}
\frac{W_{n,k+1}(a,b,c,q)}{W_{n,k}(a,b,c,q)}&=\prod_{j=1}^{r}
\left[\frac{(q^{b_{i}};q)_{k+1}}{(q^{b_{j}};q)_{k}}\right].
\left[\frac{(q^{b_{j}};q)_{n-k-1}}{(q^{b_{j}};q)_{n-k}}\right].
\left[\frac{(q^{b_{j}-c_{j}};q)_{k}}{(q^{b_{j}};q)_{k+1}}\right].
\left[\frac{(q^{b_{j}+c_{j}};q)_{n-k}}{(q^{b_{j}+c_{j}};q)_{n-k-1}}\right]=\\
&=\prod_{j=1}^{r}\left[\frac{1-q^{b_{j}+k}}{1-q^{b_{j}-c_{j}-k}}\right].\left[\frac{1-q^{b_{j}+c_{j}+n-k-1}}
{1-q^{b_{j}+n-k-1}}\right].
\end{split}
\end{equation*}
\end{proof}
Since $0<q<1$ and $b_{j}>1$ for $j=1,...,r$ we conclude that $(W_{n,k})_{k}$ is increasing and consequently $\left(C_{n}=\frac{A_{n}}{B_{n}}\right)_{n\geq0}$ is increasing too by the Lemma \ref{l1}. Thus the function $x\longmapsto h_{r}(a,b,c,q)$ is increasing on $[0,1[$ by the Lemma \ref{l2}. Therefore the inequality (\ref{110}) follows immediately from the monotonicity of the function $h_r(a,b,c,q).$

There are applications of considered inequalities in the theory of transmutation operators for estimating transmutation kernels and norms (\cite{Sit3}--\cite{Sit5}) and for problems of function expansions by systems of integer shifts of Gaussians  (\cite{Sit6}--\cite{Sit7}).

\end{document}